\documentclass[12pt]{amsart}

\numberwithin{equation}{section}
\newtheorem{thm}{Theorem}[section]
\newtheorem{prop}[thm]{Proposition}
\newtheorem{coro}[thm]{Corollary}

\newtheorem{remark}{Remark}[section]
\newtheorem{definition}{Definition}

\newenvironment{Eq}{\begin{equation}\begin{aligned}}{\end{aligned}\end{equation}\ignorespacesafterend}
\newenvironment{Eq*}{\begin{equation*}\begin{aligned}}{\end{aligned}\end{equation*}\ignorespacesafterend}

\DeclareMathOperator{\supp}{supp}

\usepackage{amsmath, amsfonts, amssymb}
\usepackage{enumitem,dsfont}
\usepackage{color}
\usepackage{graphicx}

\newtheorem{lemma}{Lemma}

\numberwithin{equation}{section}
\numberwithin{thm}{section}
\numberwithin{lemma}{section}

\title[nonlinear wave equations with singular potential]{Global solvabilty for nonlinear wave equations with singular
potential}

\begin{document}

\author{Vladimir Georgiev}
\address{V. Georgiev
\newline
Dipartimento di Matematica, Universit\`a di Pisa
Largo B. Pontecorvo 5, 56100 Pisa, Italy,
 and
 Faculty of Science and Engineering, Waseda University
 3-4-1, Okubo, Shinjuku-ku, Tokyo 169-8555,
Japan and IMI--BAS, Acad.
Georgi Bonchev Str., Block 8, 1113 Sofia, Bulgaria
}%
\email{georgiev@dm.unipi.it}%
\thanks{ The first  author was supported in part by  INDAM, GNAMPA - Gruppo Nazionale per l'Analisi Matematica, la Probabilita e le loro Applicazioni, by Institute of Mathematics and Informatics, Bulgarian Academy of Sciences, by Top Global University Project, Waseda University and  the Project PRA 2018 49 of  University of Pisa and by the project PRIN  2020XB3EFL funded by the Italian Ministry of Universities and Research.}
\author
{Hideo Kubo}
\address{Hideo Kubo \newline
Department of Mathematics,
Faculty of Science, Hokkaido University,
Sapporo 060-0810, Japan}
\email{kubo@math.sci.hokudai.ac.jp}
\thanks{The second author was
partially supported by Grant-in-Aid for Science Research (No.16H06339 and No. 19H01795), JSPS}

\keywords{Strauss conjecture, wave equation with potential, semilinear wave equation}

\subjclass[2000]{35L71, 35B33, 35L81}

\maketitle

\begin{abstract}
  In this work we study the global existence for 3d semilinear wave equation with non-negative potential satisfying generic decay assumptions.
  In the supercritical case $p > 1 + \sqrt{2}$ we establish the small data global existence result. The approach is based on appropriate conformal energy estimate in combination with Hardy inequality for conformal energy on space - like surfaces.
\end{abstract}

\section{Introduction}

In this article we study global solvabilty for the Cauchy problem
to the power-type nonlinear wave equations with the potential:
\begin{equation}\label{pn}
\begin{aligned}
 & (\partial_t^2-\Delta+V(|x|)) u = b|u|^{p-1}u,
      \quad \mbox{for }\ (t,x) \in [0,T)\times \mathbb{R}^3,
\\
 & u(0,x)=f(x), \ (\partial_t u)(0,x)=g(x) \quad \mbox{for}\ x \in {\mathbb R}^3.
\end{aligned}
\end{equation}
Here we shall assume that the potential $V$ is of the form
\begin{equation}\label{eq.int01}
    V(r) = \frac{\chi(r)}{r^2},
\end{equation}
 where $\chi(r) \in C([0,\infty)) \cap C^1((0,\infty))$ is a non - increasing bounded function satisfying
\begin{equation}\label{eq.int01m}
   \chi(r)\geq c > \frac{3}{4}.
\end{equation}

The assumption \eqref{eq.int01m} guarantees that  the operator
$$ A = -\Delta + V,$$
can be extended as an  essentially self - adjoint operator (using the result of Simon in \cite{S73}).
It is easy to see that  $V(r) \in C^1(0,\infty)$ and  satisfies
the property
\begin{equation}\label{eq.in1}
   r V^\prime(r)+2V(r) \leq  0, \ \ \forall r>0.
\end{equation}

The potentials  satisfying \eqref{eq.in1} satisfy also
\begin{equation}\label{eq.in1a1}
    V(r) \geq \frac{c}{r^2},  \  r \in (0,\infty).
\end{equation}

In particular $V(r) = a/r^2$, i.e. inverse square potential with $a >0$ satisfies \eqref{eq.in1} with equality in the place of inequality.

In this work we shall assume $b \in {\mathbb R}$, $p>1$, and $u(t,x)$ is a real-valued unknown function.
If we use the Friedrich's extension of the operator $-\Delta$ restricted to $C_0^\infty(\mathbb{R}^3 \setminus 0),$ then it is well known  that
we have  the following property  of the domain ${\mathcal D}(\sqrt{A_F}),$  where $A_F$ is the Friedrichs extension of  $-\Delta$ restricted to $C_0^\infty(\mathbb{R}^3 \setminus 0)$
\begin{equation}\label{HI0}
     {\mathcal D}(\sqrt{A_F}) \subset H^1(\mathbb{R}^3).
\end{equation}
In fact, introducing the quadratic form
 $$
q(u,v):=(\sqrt{-\Delta}u, \sqrt{-\Delta}v)_{L^2({\mathbb R}^3)}
+(V u, v)_{L^2({\mathbb R}^3)}
$$ on $C_0^\infty(\mathbb{R}^3 \setminus 0)$  and  using the Hardy inequality
\begin{equation}\label{HI1}
  (r^{-2}u,u)_{L^2({\mathbb R}^3)} \leq 4 (-\Delta u, u)_{L^2(\mathbb{R}^3)},
\end{equation}
one can see that any Cauchy sequence $u_j$ (with respect to $q$) of functions in $C_0^\infty(\mathbb{R}^3 \setminus 0)$ is a Cauchy sequence in $H^1$ so we have  \eqref{HI0}.
$$
q(u,v):=(\sqrt{-\Delta}u, \sqrt{-\Delta}v)_{L^2({\mathbb R}^3)}
+(V u, v)_{L^2({\mathbb R}^3)}
$$
is well defined on $H^1(\mathbb{R}^3)$ and this space coincides with the closure of $C_0^\infty(\mathbb{R}^3 \setminus 0)$ functions with respect to the norm generated by $q.$

We are interestead in the existence of strong solutions for \eqref{pn}, that is,
we look for a solution of \eqref{pn} in the class
$$
u\in C^2([0,T):L^2({\mathbb R}^3))
\cap C([0,T):{\mathcal D}(A))
$$
with $T>0$.

Then we may associate \eqref{pn} with the following integral equation:
\begin{Eq}\label{IE}
u(t)=(\cos t \sqrt{A} ) f+\frac{\sin t\sqrt{A}}{\sqrt{A}} g
+\int_0^t \frac{\sin (t-\tau) \sqrt{A}}{\sqrt{A}} F(u(\tau)) d\tau
\end{Eq}
with $F(u(t))=b|u(t)|^{p-1}u(t)$.


When $V\equiv 0$, the problem \eqref{pn} has been well studied.
It was shown in  \cite{John} that if $p>1+\sqrt{2}$, then the problem admits a
unique global solution for sufficiently small initial data,
and that if $1<p<1+\sqrt{2}$, then the classical solution blows up in finite time,
even though we have small initial data. In the critical case $p=1+\sqrt{2}$,
the latter blow - up behavior occur due to \cite{schaeffer2}. The blow up result for subcritical case  in higher dimensional case is due to Sideris \cite{S84} . The critical case is studied in \cite{YZ06} , where blow up result is obtained (see \cite{Yi07}, \cite{KTW12}, \cite{LY14}).  The global existence for supercritical case in the case of dimension $n \geq 3$ is obtained in \cite{GLS97} (see also \cite{K96}, \cite{KK98},\cite{T01}, \cite{G05} ).

On the other hand, the case of $V \not\equiv  0$ is less studied.
When $V$ is non - negative and satisfies $|V(x)| \lesssim (1+|x|)^{-2-\delta}$
for large values of $x$ with some $\delta>0$, the blow-up result was obtained
by \cite{YZ05} provided $1<p<1+\sqrt{2}$ (see also \cite{ST97} for the case of negative potential). In the case of smooth compactly supported potential the supercritical case was studied in \cite{GHK01}, for the case of sign - changing potential related results can be found in \cite{K06}.
In \cite{Burq2003}
it was shown among other things that if $p\ge 3$, $V(x)=a|x|^{-2},$  $a>-1/4$,
then the small amplitude solution exists globally in time (see also \cite{Plan2003}).
Recently, this result is extended by \cite{Dai} to the case $p>1+\sqrt{2}$,
provided that the solution has rotational invariance (see also \cite{Miao2013}).
We note that in these works the higher space dimensional case are also handled.

 The global existence  result is complementary to the blow - up and life - span results obtained in \cite{LLTW}.
In order to reach the critical exponent $1+\sqrt{2}$, we need to obtain suitable decay away from the light cone. However, we have no good representation formula, because we deal with the potential which is not exactly the inverse-square potential.
For this reason, we make use of the framework given by Lai \cite{Lai}.


We denote $B(k):=\{x \in {\mathbb R}^3|\, |x| \le k\}$.

\begin{thm} \label{main}
Assume $(f,g)\in {\mathcal D}(A)\times {\mathcal D}(A^{\frac12})$ satisfy
$\supp f$, $\supp g\subset B(1)$.
Suppose that $V$ is a non-negative function that satisfies \eqref{eq.in1}
and \eqref{eq.in1a1}.
Set
\begin{Eq}\label{eta}
\eta:=
\sum_{j=0}^2
\|A^{\frac{2-j}2} f\|_{L^2({\mathbb R}^3)}
+\sum_{j=1}^2
\|A^{\frac{2-j}2} g\|_{L^2({\mathbb R}^3)}.
\end{Eq}
If $1+\sqrt{2}<p<3$, then there exists $\eta_0>0$
such that for any $\eta\in (0,\eta_0)$
there exists a unique strong solution of \eqref{pn} in $[0,\infty)$ 
satisfing $\supp u(t,\cdot) \subset B(1+t)$ and
\begin{Eq*}
 \sum_{j=0}^2
\|A^{\frac{2-j}2}u(t)\|_{L^2({\mathbb R}^3)}
+\sum_{j=1}^2
\|A^{\frac{2-j}2} \partial_t u(t)\|_{L^2({\mathbb R}^3)}
\lesssim \eta
\end{Eq*}
for $t \in [0,\infty)$.
\end{thm}

In the section 2, we prepare a couple of inequalities which are valid for general functions.
Those estimates are simple generalization of \cite{Lai} for any dimensions larger than 3.
We give a new characterization of the domain of the Friedrics extension, especially \eqref{upper} in the section 3. This estimate is crucial in proving the local existence result (see Theorem 5.1 below). In the section 4, we derive a weighted $L^2$-estimate for strong solutions to the wave equation with the singular potential by using the same multiplier as in \cite{Lai}.
In the section 5, we give a proof of Theorem 1.1, by showing a blowup criterion and a suitable apriori estimate.

\section{Preliminaries}

First of all, we prepare the following estimate, in order to prove Proposition 2.2.

\begin{lemma} \label{L2L2}
Let $n \ge 2$, $s\not=1$ 
and $k >0$.
If $\phi \in C^1({\mathbb R}^n)$ satisfies $\supp \phi \subset B(k)$, then we have
\begin{Eq*}
& \|(1+k-r)^{(s-2)/2} \phi\|_{L^2 ({\mathbb R}^n) }
\\
\lesssim \, & \|(1+k-r)^{s/2} \partial_r \phi\|_{L^2 ({\mathbb R}^n) }
+\|(1+k-r)^{s/2} \frac{\phi}{r}\|_{L^2 ({\mathbb R}^n) },
\end{Eq*}
as long as the right hand side is finite.
Here the implicit constant is independent of $k$.
\end{lemma}

\begin{proof}
It follows that
\begin{Eq*}
& \|(1+k-r)^{(s-2)/2} \phi\|_{L^2 ({\mathbb R}^n) }^2
\\
=\, &
\frac1{s-1} \int_{S^{n-1}} dS_\omega \int_0^k
(1+k-r)^{s-1} \partial_r ( (\phi (r\omega))^2 r^{n-1} ) dr
\\
\le \, &
\frac{2}{s-1}
\|(1+k-r)^{s/2} \partial_r \phi\|_{L^2 ({\mathbb R}^n) }
\|(1+k-r)^{(s-2)/2} \phi\|_{L^2 ({\mathbb R}^n) }
\\
& \
+\frac{n-1}{s-1}
\|(1+k-r)^{s/2} \frac{\phi}{r}\|_{L^2 ({\mathbb R}^n) }
\|(1+k-r)^{(s-2)/2} \phi\|_{L^2 ({\mathbb R}^n) },
\end{Eq*}
which leads to the conclusion.
\end{proof}

\begin{prop} \label{Lai}
Let $n \ge 2$, $s\not=1$ 
and $k \ge 1$.
If $\phi \in C^1({\mathbb R}^n)$ satisfies $\supp \phi \subset B(k)$, then we have
\begin{Eq*}
& \|r^{(n-2)/2} (1+k+r)^{1/2}(1+k-r)^{(s-1)/2} \phi(r\omega)\|_{L^\infty_r L^2(S^{n-1})}
\\
\lesssim \, & \|(1+k-r)^{s/2} \partial_r \phi\|_{L^2 ({\mathbb R}^n) }
+\|(1+k-r)^{s/2} \, \frac{\phi}{r}\|_{L^2 ({\mathbb R}^n) }
\end{Eq*}
Here the implicit constant is independent of $k$.
\end{prop}

\begin{proof}
Let $x \in {\mathbb R}^n$.
First, suppose $0 \le 2r \le 1+k$, so that $1+k+r$ is equivalent to
$1+k-r$. Then, it suffices to show
\begin{Eq*}
& r^{n-2}(1+k-r)^{s} \|\phi(r\omega)\|_{L^2(S^{n-1})}^2
\\
\lesssim \, & \|(1+k-r)^{s/2} \partial_r \phi\|_{L^2 ({\mathbb R}^n) }^2
+\|(1+k-r)^{s/2} \, \frac{\phi}{r}\|_{L^2 ({\mathbb R}^n) }^2.
\end{Eq*}
Rewriting the left hand side, we see that it equals to
\begin{Eq*}
& -\int_r^k \partial_\lambda
(\lambda^{n-2} (1+k-\lambda)^{s} \|\phi(\lambda \omega)\|_{L^2(S^{n-1})}^2)
d\lambda
\\
\le \, &
s \int_0^\infty \lambda^{n-2} (1+k-\lambda)^{s-1} \|\phi(\lambda \omega)\|_{L^2(S^{n-1})}^2
d\lambda
\\
& \
+2 \int_0^\infty \lambda^{n-2} (1+k-\lambda)^{s}
\left( \int_{S^{n-1}} \phi(\lambda \omega)
  \partial_r \phi(\lambda \omega) dS_{\omega}
  \right)
d\lambda
\\
\lesssim \, &
\|(1+k-r)^{s/2} \, \frac{\phi}{r}\|_{L^2 ({\mathbb R}^n) }
\|(1+k-r)^{(s-2)/2} \phi\|_{L^2 ({\mathbb R}^n) }
\\
& \
+\|(1+k-r)^{s/2} \, \frac{\phi}{r}\|_{L^2 ({\mathbb R}^n) }
\|(1+k-r)^{s/2} \partial_r \phi\|_{L^2 ({\mathbb R}^n) }.
\end{Eq*}
Using Lemma \ref{L2L2} in order to bound the first term, we obtain the desired estimate.

Next, suppose $1+k \le 2r $, so that $1+k+r$ is equivalent to
$r$.
Similarly to the previous case, we have
\begin{Eq*}
& r^{n-1}(1+k-r)^{s-1} \|\phi(r\omega)\|_{L^2(S^{n-1})}^2
\\
=\, &
 -\int_r^k \partial_\lambda
(\lambda^{n-1} (1+k-\lambda)^{s-1} \|\phi(\lambda \omega)\|_{L^2(S^{n-1})}^2)
d\lambda
\\
\le \, &
(s-1) \|(1+k-r)^{(s-2)/2} \phi\|_{L^2 ({\mathbb R}^n) }^2
\\
& \
+2 \|(1+k-r)^{s/2} \partial_r \phi\|_{L^2 ({\mathbb R}^n) }
\|(1+k-r)^{(s-2)/2} \phi\|_{L^2 ({\mathbb R}^n) },
\end{Eq*}
which yields the desire estimate.
\end{proof}

Next we introduce the following analogous estimate to the Hardy inequality.

 \begin{lemma}
Let $s\ge 0$ and $k \ge 1$.
If $\phi \in C^1({\mathbb R}^3)$ satisfies $\supp \phi \subset B(k)$, then we have
\begin{eqnarray} \label{hardy1}
& \quad \|(1+k-r)^{s/2} \, \displaystyle\frac{\phi}{r} \|_{L^2 ({\mathbb R}^3)}
\le 2
\|(1+k-r)^{s/2} (\partial_r \phi + \displaystyle\frac{\phi}{r})\|_{L^2 ({\mathbb R}^3)},
\\  \label{hardy2}
& \quad \|(1+k-r)^{s/2} \, \partial_r \phi \|_{L^2 ({\mathbb R}^3)}
\le
\|(1+k-r)^{s/2} (\partial_r \phi + \displaystyle\frac{\phi}{r})\|_{L^2 ({\mathbb R}^3)}.
\end{eqnarray}
\end{lemma}

\begin{proof}
First we prove \eqref{hardy1}.
It follows that
\begin{Eq*}
& \int_0^\infty(1+k-r)^s\phi^2dr
\\
=& s\int_0^\infty(1+k-r)^{s-1}\phi^2rdr
-2\int_0^\infty(1+k-r)^s\phi\, \partial_r\phi \, rdr
\\
\ge &
-2\int_0^\infty(1+k-r)^s \phi\,
(\partial_r \phi + \displaystyle\frac{\phi}{r})\, rdr
+2\int_0^\infty(1+k-r)^s\phi^2dr,
\end{Eq*}
since $s\ge 0$.
Therefore, we have
\begin{Eq*}
\int_0^\infty\!\!\int_{s^2}(1+k-r)^s\phi^2dS_\omega dr
\le 2\int_0^\infty\!\!\int_{S^2}
(1+k-r)^s \phi\,
(\partial_r \phi + \displaystyle\frac{\phi}{r})\, r dS_\omega dr,
\end{Eq*}
which leads to \eqref{hardy1} by the Schwarz inequality.

Next we prove \eqref{hardy2}
It follows that
\begin{Eq*}
& \int_0^\infty \!\! \int_{S^2}
(1+k-r)^s  (\partial_r \phi +\frac{\phi }{r})^2 r^2
 dS_\omega dr
 \\
 = & \,
\int_0^\infty \!\! \int_{S^2}
(1+k-r)^s  (\partial_r \phi)^2 r^2
 dS_\omega dr
 \\
 &+\int_0^\infty \!\! \int_{S^2}
(1+k-r)^s  \partial_r (\phi^2) \, r
 dS_\omega dr
 +\int_0^\infty\!\!\int_{S^2}(1+k-r)^s\phi^2dS_\omega dr
 \\
 = & \, \int_0^\infty \!\! \int_{S^2}
(1+k-r)^s  (\partial_r \phi )^2 r^2
 dS_\omega dr
\\ & \,
+ s \!\int_{S^2}\left(  \int_0^\infty
(1+k-r)^{s-1} \phi ^2 r dr \right)
 dS_\omega,
\end{Eq*}
which implies the desired estimate, because $s\ge 0$.
\end{proof}

\section{Representation of the domain of the Friedrichs extension}

Our starting point is to introduce the operator
$A_0 = -\Delta + \frac{\chi(r)}{r^2}$ with domain
$\mathcal{D}(A_0) = C_0^\infty(\mathbb{R}^3 \setminus 0).$
In the case when the assumption \eqref{eq.int01m} is satisfied, i.e.
\begin{equation}\label{eq.int01mm1}
   \chi(r)\geq c > \frac{3}{4}.
\end{equation}
the important result of Simon in \cite{S73} states that $A_0$ is essentially self-adjoint. As (unique) self - adjoint extension we can choose the Friedrichs extension $A_F$ of $A_0.$

In fact, introducing the quadratic form
 $$
q(u,v):=(\sqrt{-\Delta}u, \sqrt{-\Delta}v)_{L^2({\mathbb R}^3)}
+(V u, v)_{L^2({\mathbb R}^3)}
$$ on $C_0^\infty(\mathbb{R}^3 \setminus 0)$    and  using the Hardy inequality
\begin{equation}\label{HI1m}
  (r^{-2}u,u)_{L^2({\mathbb R}^3)} \leq 4 (-\Delta u, u)_{L^2(\mathbb{R}^3)},
\end{equation}
one can see that any Cauchy sequence $u_j$ (with respect to $q$)  of functions in $C_0^\infty(\mathbb{R}^3 \setminus 0)$ is a Cauchy sequence in $H^1$ so we have  \begin{equation}\label{HI0m}
     {\mathcal D}(q) \subset H^1(\mathbb{R}^3),
\end{equation}
where the domain  ${\mathcal D}(q) $ is  completion of $C_0^\infty(\mathbb{R}^3 \setminus 0)$ with respect to $q.$
The domain  ${\mathcal D}(q) $  is simply
\begin{equation}\label{HI0m1}
     {\mathcal D}(q) =  H^1(\mathbb{R}^3).
\end{equation}

To verify this it is sufficient to show that any $C_0^\infty(\mathbb{R}^3)$ function $f(x)$
can be approximated in $H^1(\mathbb{R}^3)$ by functions in $C_0^\infty(\mathbb{R}^3 \setminus 0).$
For the purpose we define $\varphi \in C_0^\infty(\mathbb{R}^3)$ such that
$$ \varphi(x)= \left\{
\begin{array}{cc}
  1   & \mbox{if $|x|\leq 1$} \\
   0  & \mbox{if $|x| \ge 2.$}
\end{array}
\right.
$$
Now we can construct easily the needed approximating  sequence
$$f_k(x) = (1-\varphi(kx))f(x) \in C_0^\infty(\mathbb{R}^3 \setminus 0), \ k \in \mathbb {N}. $$ Indeed, we have the pointwise estimates
$$| f_k(x) - f(x) | \lesssim |\varphi(kx)|,\ \  |\nabla f_k(x) - \nabla f(x) | \lesssim k |(\nabla \varphi)(kx)| + |\varphi(kx)| $$
and use the fact that
$$ \|\varphi(kx)\|_{L^2(\mathbb {R}^3)} + k\|(\nabla \varphi)(kx)\|_{L^2(\mathbb {R}^3)} \to 0$$
as $k \to \infty.$

Recall that the Friedrichs extension $A_F$ has domain
\begin{equation}\label{eq.FE1}
    \mathcal {D} (A_F) = \left\{ f \in   {\mathcal D}(q) , \mbox {$g \in \mathcal{D}(q) \to q(f,g)$ is bounded in $L^2$} \right\}.
\end{equation}
More precisely, the $L^2$ boundedness means that for fixed $f \in  \mathcal {D}(q)$ we have
$$
|g(f,g)| \le C_f \|g \|_{L^2(\mathbb {R}^3)}.
$$
Then the  Riesz representation theorem implies that $A_Ff$ is well defined, since  $$ q(f,g) = (A_Ff,g)_{L^2} $$
and the Friedrichs theorem guarantees that $A_F$ is non negative self - adjoint operator with
$$ \mathcal {D}(\sqrt{A_F}) = \mathcal {D}(q).  $$

The above observations lead to the following well -known conclusion.

\begin{lemma} \label{l.daf1m}
 The operator
$A_0 = -\Delta + \frac{\chi(r)}{r^2}$ with $\chi(r) \geq c > - 1/4$  and with domain
$\mathcal{D}(A_0) = C_0^\infty(\mathbb{R}^3 \setminus 0)$ has a Friedrichs extension $A_F,$ such that
\begin{description}
  \item[i)] we have $$ \mathcal {D}(\sqrt{A_F}) = \mathcal {D}(q) = H^1(\mathbb{R}^3); $$
  \item[ii)] for $u \in \mathcal{D}(\sqrt{A_F})$  we have \begin{equation}\label{equivm1}
\|\sqrt{A_F}\, u\|_{L^2( \mathbb{R}^3)}\sim \|\nabla u\|_{L^2( \mathbb{R}^3)}.
\end{equation}
\end{description}
\end{lemma}

Our next step is to give appropriate characterizations of the domains ${\mathcal D}(A_F)$  in the case when $c>3/4.$
In this case, the result of Simon \cite{S73} guarantees that $A_F$ coincides with the graph closure $\overline{A_0}$ of $A_0$, i.e.

\begin{definition}
A function
$u \in L^2$ belongs to $ \mathcal {D} (\overline{A_0}) $ if and only if there exists a sequence $\{u_k\}_{k \in \mathbb{N}}$ in $C_0^\infty(\mathbb {R}^3\setminus 0)$ and a function $g \in L^2$ so that
\begin{equation}\label{eq.FE10}
\begin{aligned}
 \lim_{k \to \infty}\|u_k - u\|_{L^2} = 0, \\
 \lim_{k \to \infty}\|A_0(u_k) - g\|_{L^2} = 0.
\end{aligned}
\end{equation}
\end{definition}

We shall give now simple characterization of the unique symmetric closure $\overline{A_0}$ of the operator
$A_0 = -\Delta + \frac{\chi(r)}{r^2}$ with domain
$\mathcal{D}(A_0) = C_0^\infty(\mathbb{R}^3 \setminus 0)$ in the case
\begin{equation}\label{sas.1}
    \inf \chi = c > \frac{3}{4}.
\end{equation}

\begin{lemma} \label{l.daf1}
Assume \eqref{sas.1} is fulfilled. Then we have
\begin{description}
  \item[a)] $${\mathcal D}(\overline{A_0}) = \mathcal{D}(A_F) = \{u \in H^2(R^3); \ \ u(0)=0 \},$$
  \item[b)] for $u\in {\mathcal D}(\overline{A_0})$ we have
  \begin{Eq}\label{upper}
\|\Delta  u\|_{L^2( \mathbb{R}^3)}
\lesssim
\|A_F u\|_{L^2( \mathbb{R}^3)}.
\end{Eq}
\end{description}
\end{lemma}

\begin{proof}
Take sequence $\{u_k\}_{k \in \mathbb{N}}$ in $C_0^\infty(\mathbb {R}^3\setminus 0)$ and set
$$ g_k = -\Delta u_k + \frac{\chi(r)u_k}{r^2}.$$
Using radial coordinates $r=|x|, \omega = x/|x|$ we rewrite the  above relation as equation
$$ -r^{-2}\partial_r( r^2 \partial_r) u_k - \frac{1}{r^2} \Delta_{\mathbb{S}^2}u_k + \frac{\chi(r)u_k}{r^2} = g_k.$$
We multiply this equation by $u_k(r\omega)$ and integrate over $r \in (0,\infty),$ $\omega \in \mathbb{S}^2$ using the measure $dr d\omega$
\begin{equation}\label{eq.miid1}
   \left\{ \begin{aligned}
    &\int_0^\infty \int_{\mathbb{S}^2} |\partial_r u_k(r\omega)|^2  d\omega dr -  \int_0^\infty \int_{\mathbb{S}^2} \frac{| u_k(r\omega)|^2}{r^2}  d\omega dr + \\
    & + \int_0^\infty \int_{\mathbb{S}^2} \frac{|\nabla_\omega u_k(r\omega)|^2}{r^2}  d\omega dr + \\
    & + \int_0^\infty \int_{\mathbb{S}^2} \frac{\chi(r) | u_k(r\omega)|^2}{r^2}  d\omega dr = \int_0^\infty \int_{\mathbb{S}^2} g_k(r\omega) u_k(r\omega)  d\omega dr,
    \end{aligned}\right.
\end{equation}
where in the integration by parts we have used the fact that $u(r\omega)=0$ for $r$ close to $0.$ Using the Hardy inequality
$$ \int_0^\infty |\partial_r f(r)|^2 dr \geq \frac{1}{4} \int_0^\infty \frac{| f(r)|^2}{r^2} dr $$
valid when $f(r)$ is $C^1$ function on $(0,\infty)$ with compact support   in $(0,\infty)$ we see that the left side of \eqref{eq.miid1} can be estimated from below by
$$ \left(c -\frac{3}{4} \right) \int_0^\infty \int_{\mathbb{S}^2} \frac{ | u_k(r\omega)|^2}{r^2}  d\omega dr . $$
We can bound the right side \eqref{eq.miid1} from above by using Cauchy inequality by
$$ \|g_k\|_{L^2(\mathbb{R}^3)}  \left( \int_0^\infty \int_{\mathbb{S}^2} \frac{ | u_k(r\omega)|^2}{r^2}  d\omega dr
 \right)^{1/2}. $$ Hence we have the estimate
$$\left(c -\frac{3}{4} \right) \left(\int_0^\infty \int_{\mathbb{S}^2} \frac{ | u_k(r\omega)|^2}{r^2}  d\omega dr \right)^{1/2} \leq \|g_k\|_{L^2(\mathbb{R}^3)}.$$
Therefore, we get
\begin{equation}\label{eq.cm341}
   \left(c -\frac{3}{4} \right) \left\| \frac{u_k}{|\cdot|^2} \right\|_{L^2(\mathbb{R}^3)}  \leq \|g_k\|_{L^2(\mathbb{R}^3)}.
\end{equation}
This estimate is valid replacing $u_k$ by $u_k-u_m$ so using the fact that $g_k$ is a Cauchy sequence in $L^2(\mathbb{R}^3)$, we see that $u_k(x)/|x|^2$ is a Cauchy sequence in $L^2(\mathbb{R}^3)$ so $u/|x|^2 \in L^2(\mathbb{R}^3) $ and from the equation
\begin{equation}\label{eq.eel1}
    -\Delta u + \frac{\chi(r)u}{r^2}  = g
\end{equation}
we deduce $u \in H^2(\mathbb{R}^3).$ In this case $u(0)$ is well defined and must be $0$ since
$$ \left\| \frac{u}{|\cdot|^2} \right\|_{L^2(\mathbb{R}^3)} < \infty.$$

In conclusion, we have established that
$${\mathcal D}(\overline{A_0}) = \mathcal{D}(A_F) \subseteq  \{u \in H^2(R^3); \ \ u(0)=0 \}$$
and (taking the limit $k \to \infty$ in \eqref{eq.cm341})
\begin{equation}\label{eq.cm342}
   \left(c -\frac{3}{4} \right) \left\| \frac{u}{|\cdot|^2} \right\|_{L^2(\mathbb{R}^3}  \leq \|g\|_{L^2(\mathbb{R}^3)}
\end{equation}
for any $u \in \mathcal {D}(\overline{A_0}).$

The opposite inclusion easily follows from Hardy inequality.
This completes the proof of a).
The proof of b) follows directly from \eqref{eq.cm342} and the equation \eqref{eq.eel1}.

This completes the proof.
\end{proof}

\section{Weighted energy estimates}

We plan to prove a  weighted $L^2$ - estimate of conformal type.
In principle the conformal type estimate can be derived using appropriate multipliers for
the linear wave equation
\begin{equation}\label{eq.wee1}
   \partial_t^2 u - \Delta u + Vu = F,
\end{equation}
assuming 
$F \in L^1( (0,\infty);L^2(\mathbb{R}^3))$
are supported in $\{|x| \leq t +1 \}$ for any $t \geq 0.$

We use spherical harmonics $Y^k_\ell, k \in \mathbb{Z}, |k| \leq \ell$ on the sphere $\mathbb{S}^2$ that are solutions to the equation
$$ -\Delta_{S^2} Y^k_\ell = \lambda_\ell Y^k_\ell,\  \ \lambda_\ell =  \ell (1+\ell), \ \ \ \ell = 0,1,2,\cdots. $$
Using radial coordinates $r=|x|, \omega=x/r$ and the expansions of $u,F$
\begin{equation}\label{eq.sh1}
\left\{
   \begin{aligned}
  & u(t,x) = \sum_{\ell=0}^\infty \sum_{k=-\ell}^\ell \frac{u^k_\ell(t,r)}{r} Y^k_\ell(\omega), \\ & F(t,x) = \sum_{\ell=0}^\infty \sum_{k=-\ell}^\ell \frac{F^k_\ell(t,r)}{r} Y^k_\ell(\omega),
   \end{aligned}
   \right.
\end{equation}
one can extend $u^k_\ell(t,r), F^k_\ell(t,r)$ as odd functions on $r \in \mathbb{R}$ and rewrite \eqref{eq.wee1} as
\begin{equation}\label{eq.wee4}
   (\partial_t^2 - \partial_r^2) u_\ell^k(t,r) +\left( V(r)+ \frac{\lambda_\ell}{r^2} \right) (u_\ell^k) = F^k_\ell(t,r).
\end{equation}

It is clear that $r^2V(r)=\chi(r)$  can be extended as even function on  $\mathbb{R}.$

To simplify the notations we shall omit the indices $k,\ell$ so we shall consider the 1D equation
\begin{equation}\label{eq.wee5-}
   (\partial_t^2 - \partial_r^2) u(t,r) +\left( V(r)+ \frac{\lambda}{r^2} \right) u(t,r) = F(t,r).
\end{equation}
Using the notations
$$ \nabla_\pm = \partial_t \pm \partial_r, \ \tau_\pm = 2+t\pm r ,$$ we can rewrite \eqref{eq.wee5-} as
\begin{equation}\label{eq.wee5}
   \nabla_+ \nabla_- u(t,r) +\left( V(r)+ \frac{\lambda}{r^2} \right) u(t,r) = F(t,r).
\end{equation}
We use as a multiplier
$$ M(u):= \tau_+^{s} \nabla_+ u
+\tau_-^{s}  \nabla_-u.$$

\begin{lemma}
Let $s \geq 0$ and $T>0$.
For $u \in C^2([0,T)\times ({\mathbb R}^3 \setminus 0) )$, we have
\begin{Eq}  \label{DF}
& M(u) \, \left[\partial_t^2 u-\partial_r^2 u+\left( V(r) + \frac{\lambda}{r^2}\right) u \right]=
\\
=\,
& \frac12(\partial_t+\partial_r )X_+u +\frac12
(\partial_t-\partial_r )X_-u
 +\mathfrak{R}|u|^2
\end{Eq}
where we put
\begin{Eq}\label{cee1}
X_\mp(u) & = \frac{1}{2}  \tau_\pm^s|\nabla_\pm u|^2 + \frac{\tau_\mp^s}{2}\left( V(r)+ \frac{\lambda}{r^2} \right)|u|^2,
\\
\mathfrak{R}(r,t) & = - \frac{(\tau_+^s-\tau_-^s)}{2r} \left(rV^\prime(r)+2V(r) \right) + \frac{\rho V}{r} +
\frac{\rho \lambda}{r^3},
\end{Eq}
with
\begin{equation}\label{wee61}
   \rho(t,r) = \tau_+^s-\tau_-^s -s(r\tau_+^{s-1}+ r \tau_-^{s-1}).
\end{equation}
\end{lemma}

\begin{proof}
Using the commutator relations
$$ [\nabla_+, \nabla_-]=0,  [\nabla_-,\tau_+]= [\nabla_+,\tau_-] =0 ,$$
as well as the relations
$$\nabla_+ \tau_+^a = 2a\tau_+^{a-1}, \nabla_- \tau_-^a = 2a\tau_-^{a-1},$$
we arrive at
$$ M(u) F = \nabla_+ \left(\frac{1}{2}  \tau_-^s|\nabla_- u|^2 \right)  + \nabla_- \left(\frac{1}{2}  \tau_+^s|\nabla_+ u|^2\right) + $$ $$ +  \frac{\tau_+^s}{2}\left( V(r)+ \frac{\lambda}{r^2} \right) \nabla_+ |u|^2 +\frac{\tau_-^s}{2}\left( V(r)+ \frac{\lambda}{r^2} \right) \nabla_- |u|^2,$$
where $F$ represents the left-hand side of \eqref{eq.wee5}.
Since 
$$ X_\mp(u) = \frac{1}{2}  \tau_\pm^s|\nabla_\pm u|^2 + \frac{\tau_\mp^s}{2}\left( V(r)+ \frac{\lambda}{r^2} \right)|u|^2,$$
we obtain
$$   M(u) F = \nabla_+ X_+(u) +\nabla_- X_-(u)  + \mathfrak{R}|u|^2,$$ where
$$ \mathfrak{R} = - \nabla_- \left( \frac{\tau_-^s}{2}\left( V(r)+ \frac{\lambda}{r^2} \right) \right)-
\nabla_+ \left( \frac{\tau_+^s}{2}\left( V(r)+ \frac{\lambda}{r^2} \right) \right)$$
$$ = -(s \tau_-^{s-1} + s \tau_+^{s-1}) \left( V(r)+ \frac{\lambda}{r^2} \right) + \frac{(\tau_-^s-\tau_+^s)}{2} \left( V^\prime(r)- \frac{2\lambda}{r^3}\right)=$$
$$ = - \frac{(\tau_+^s-\tau_-^s)}{2r} \left(rV^\prime(r)+2V(r) \right) + \frac{\rho V}{r} +
\frac{\rho \lambda}{r^3}.
$$
This  completes the proof.
\end{proof}

\begin{remark}
It is easy to show that
\begin{equation}\label{eq.ron1}
    \rho(r,t) \geq 0 , \ \ \forall r\in [0,t+1], \ \forall s \in [1,2].
\end{equation}

Indeed, we have the relation
$$ \rho(r,t) = (2+t)^s f \left( \frac{r}{2+t}  \right),$$
where
 $$f(x):=\underbrace{(1+x)^s -sx(1+x)^{s-1}}_{= \phi(x)} - \underbrace{\left((1-x)^s
+
 sx(1-x)^{s-1}\right)}_{\phi(-x)}.  $$

Note that the derivative of $f$ is
$$
s(1-s)x ( (1+x)^{s-2}-(1-x)^{s-2} )
$$
and it is positive for $x \in (0,1).$ Hence $f(x) \geq f(0)=0.$	
\end{remark}

\begin{prop}
Suppose that $V$ is a non-negative function that satisfies \eqref{eq.in1}
and \eqref{eq.in1a1}.
Let $1<s<2$ and $0<\delta<s-1$.
We assume that $F \in C([0,T): L^2({\mathbb R}^3))$
satisfies $\supp F(t,\cdot) \subset B(1+t)$ and
that $u$ is the strong solution to
\begin{Eq}\label{linear}
& (\partial_t^2+A)u 
= F(t,x)
      \quad \mbox{for }\ (t,x) \in [0,T)\times \mathbb{R}^3,
\\
& u(0,x)=f(x), \ (\partial_t u)(0,x)=g(x) \quad \mbox{for}\ x \in {\mathbb R}^3
\end{Eq}
satisfying $\supp u(t,\cdot) \subset B(1+t)$.
Then, for $t \in [0,T)$, we have
\begin{equation} \label{CE}
\begin{aligned}
& \| (2+t-r)^{s/2} \nabla_{t,r} u(t)\|_{L^2({\mathbb R})}
\\ & \
+\sum_{j=1}^3 \left\| (2+t-r)^{s/2} \frac{|R_j u(t)|}{r} \right\|_{L^2({\mathbb R}^3)} \\
&
+\left\| (2+t-r)^{s/2}\, \frac{u(t)}{r}\right\|_{L^2({\mathbb R}^3)}
\le
C_\delta (\|g\|_{L^2({\mathbb R}^3)}+\|\nabla f\|_{L^2({\mathbb R}^3)}
\\
&  +
\|(2+t+r)^{s/2}(2+t-r)^{(1+\delta)/2} F\|_{L^2([0,T) \times {\mathbb R}^3)}).
\end{aligned}
\end{equation}
\end{prop}

\begin{proof}

Using radial coordinates $r=|x|, \omega=x/r$ and the expansions \eqref{eq.sh1} of $u,F,f,g$ in spherical harmonics,
we reduce the proof to the following estimate for $u^k_\ell$
\begin{equation} \label{CEs1}
\left\{
\begin{aligned}
& \| (2+t-r)^{s/2} \nabla_{t,r}u^k_\ell(t)\|_{L^2({\mathbb R})}
+ \\ &+
\ell \left\| (2+t-r)^{s/2}\, \frac{u^k_\ell(t)}{r}\right\|_{L^2({\mathbb R})}
\le\\
&\le
C_\delta (\|g^k_\ell\|_{L^2({\mathbb R})}+\|\nabla f^k_\ell\|_{L^2({\mathbb R})}
\\
&  +
\|(2+t+r)^{s/2}(2+t-r)^{(1+\delta)/2} F^k_\ell\|_{L^2([0,T) \times {\mathbb R})}).
\end{aligned}
\right.
\end{equation}
Here and below $u^k_\ell$ are  solutions to  the $1D$ semilinear wave equation of type \eqref{eq.wee5}, i.e.
\begin{equation}\label{eq.wee5m}
   (\partial_t^2 - \partial_r^2) u^k_\ell(t,r) +\left( V(r)+ \frac{\lambda_\ell}{r^2} \right) u^k_\ell(t,r) = F^k_\ell(t,r)
\end{equation}
with initial data $f^k_\ell, g^k_\ell.$

To arrive at \eqref{CE} it is sufficient to take into account the Hardy type estimates
\eqref{hardy1} and see  that \eqref{CEs1} implies \eqref{CE}.

Then the identity \eqref{DF} suggests to integrate the  expression
$$ \frac12(\partial_t+\partial_r )X_+u^k_\ell +\frac12
(\partial_t-\partial_r )X_-u^k_\ell
 +\mathfrak{R}|u^k_\ell|^2 $$
 in the right side of \eqref{DF} over appropriate space time domain.
To explain more clearly the domain of integration we take $T >1$ and then choose two positive parameters $\alpha, \beta$ satisfying the relations
\begin{equation}\label{eq.pent1m}
  2T+1 > \alpha, \
  \alpha>\beta >-1,\
  \ \frac{\alpha + \beta}{2} > T.
\end{equation}

First, we consider the case
\begin{equation}\label{eq.pent1}
  2T+1 > \alpha, \
  \alpha>\beta >1,\
  \ \frac{\alpha + \beta}{2} > T.
\end{equation}
In this case the domain for $(\alpha,\beta)$ is represented by   the hexagon $\mathrm{hex}(\alpha,\beta,T)=ABCDEF$ on Figure \ref{fig:f0}.

Applying the Gauss - Green formula for the integral of \eqref{DF} over the hexagon, we find
$$ \iint_{\mathrm{hex}(\alpha,\beta,T)} M(u^k_\ell) F^k_\ell = E_-(\beta,T) + $$ $$ E_+(\alpha,T) + E_0(\alpha,\beta,T) + \iint_{\mathrm{hex}(\alpha,\beta,T)}\mathfrak{R}|u^k_\ell|^2 -\frac{1}{2} \int_{EF}X_+u^k_\ell + X_-u^k_\ell, $$
where
$$ E_-(\beta,T) = \frac{1}{\sqrt{2}}\int_{AB(\beta,T)}X_-u^k_\ell, \ E_+(\alpha,T) = \frac{1}{\sqrt{2}}\int_{CD(\beta,T)}X_+u^k_\ell$$
$$ E_0(\alpha,\beta,T) = \frac{1}{2}\int_{BC(\alpha,\beta,T)}X_+u^k_\ell + X_-u^k_\ell$$
Of special interests are the segments $AB, BC, CD$ parameterized as follows:
$$ AB(\beta): r =  \rho, t =\rho+\beta,  \ \rho \in [-\frac{\beta+1}{2},T-\beta],$$
$$ BC(\alpha,\beta) : r = \rho, t =T, \ \ \rho \in [T-\beta, \alpha -T],$$
$$ CD(\alpha) : r =\rho , t = \alpha - \rho, \ \rho \in \left[\alpha-T, \frac{\alpha + 1}{2} \right] $$
\begin{figure}[htb!]
\centering%
\includegraphics[scale=0.3]{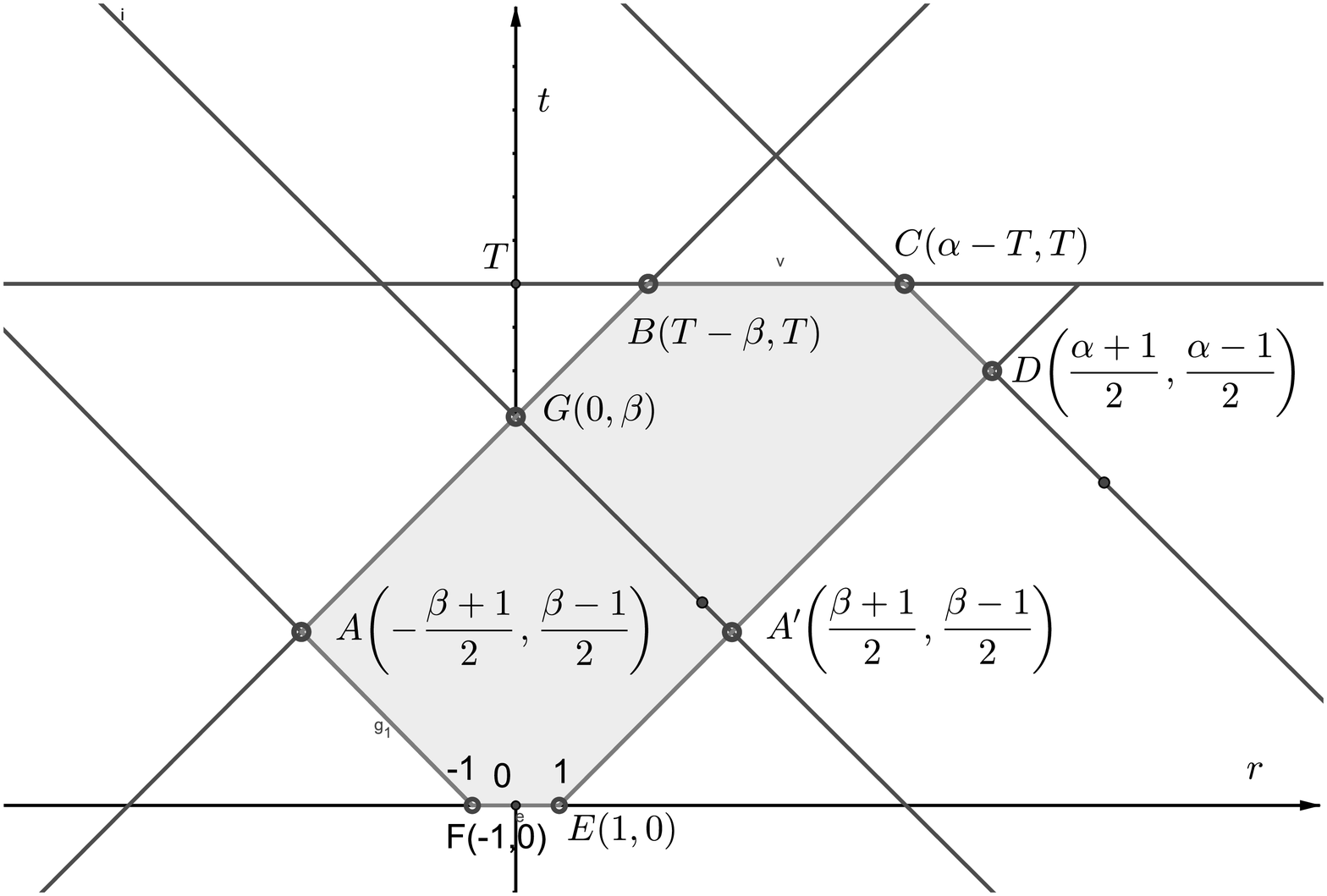}
\caption{Domain of integration for $\beta>1.$}
\label{fig:f0}
\end{figure}
We have the following estimates
$$ \left| \int_{EF}X_+u^k_\ell + X_-u^k_\ell \right| \lesssim \|g^k_\ell\|^2_{L^2({\mathbb R})}+\|\nabla f^k_\ell\|^2_{L^2({\mathbb R})}$$
and
$$ \left|\iint_{\mathrm{hex}(\alpha,\beta,T)} M(u^k_\ell) F^k_\ell \right| \leq \int_0^T \int_\mathbb{R} \left| M(u^k_\ell) F^k_\ell \right|dr dt.$$
Hence, we have the estimate
\begin{equation}\label{eq.esc1}
  \left\{  \begin{aligned}
    E_-(\beta,T) + E_+(\alpha,T) + E_0(\alpha,\beta,T) + \iint_{\mathrm{hex}(\alpha,\beta,T)}\mathfrak{R}|u^k_\ell|^2 \lesssim \\ \lesssim  \|g^k_\ell\|^2_{L^2({\mathbb R})}+\|\nabla f^k_\ell\|^2_{L^2({\mathbb R})} +   \int_0^T \int_\mathbb{R} \left| M(u^k_\ell) F^k_\ell \right|dr dt
\end{aligned}\right.
\end{equation}

Using \eqref{cee1}, we deduce the following estimates
\begin{equation}\label{eq.esc4}
  \left\{  \begin{aligned}
 &    E_-(\beta,T) \geq \int_{AB(\beta,T)} (2+t+r)^{s}  (\partial_{t} u+\partial_r u)^2, \\
& E_+(\alpha,T) \geq \int_{CD(\alpha,T)} (2+t-r)^{s}  (\partial_{t} u-\partial_r u)^2,\\
& E_0(\alpha, \beta,T) \geq \int_{BC(\alpha, \beta, T)} (2+t-r)^{s} \left( (\nabla_{t,r} u)^2 + \ell^2 \frac{|u|^2}{r^2}\right).
\end{aligned}\right.
\end{equation}
The multiplier
$$ M(u):= \tau_+^{s} \nabla_+ u
+\tau_-^{s}  \nabla_-u$$
can be substituted in the right side of \eqref{eq.esc1} so  we get
$$  \int_0^T \int_\mathbb{R} \left|\tau_+^{s/2} \nabla_+ u^k_\ell \tau_+^{s/2} F^k_\ell \right|dr dt
= \int_{-1}^{2T+1} d\beta \int_{AB(\beta,T)} \left|\tau_+^{s/2} \nabla_+ u^k_\ell \tau_+^{s/2} F^k_\ell \right| \lesssim $$
$$
\lesssim
  \int_{-1}^{2T+1} d\beta E_-(\beta,T)^{1/2} \left(\int_{AB(\beta,T)} \tau_+^{s}|F^k_\ell |^2 \right)^{1/2} \lesssim$$
$$ \lesssim
\sup_{-1<\beta<2T+1}
E_-(\beta,T)^{1/2} \left(\int_{-1}^{2T+1} d\beta\int_{AB(\beta,T)}(1+\beta)^{s} \tau_-^s |F^k_\ell |^2 \right)^{1/2}=$$
$$ =
\sup_{-1<\beta<2T+1}
E_+(\beta,T)^{1/2}
\left\| \tau_+^{s/2} \tau_-^{s/2}F^k_\ell\right\|_{L^2((0,T)\times \mathbb{R})}.$$
In a similar way we get
$$  \int_0^T \int_\mathbb{R} \left|\tau_-^{s/2} \nabla_- u^k_\ell \tau_-^{s/2} F^k_\ell \right|dr dt \lesssim
$$ $$ \lesssim
\sup_{-1<\alpha<2T+1}
E_-(\alpha,T)^{1/2}
\left\| \tau_+^{s/2} \tau_-^{s/2}F^k_\ell\right\|_{L^2((0,T)\times \mathbb{R})}.$$
Turning back  to \eqref{eq.esc1} we get
\begin{equation}\label{eq.esc2}
  \left\{  \begin{aligned}
    E_-(\beta,T) + E_+(\alpha,T) + E_0(\alpha,\beta,T) + \iint_{\mathrm{hex}(\alpha,\beta,T)}\mathfrak{R}|u^k_\ell|^2 \lesssim \\ \lesssim  \|g^k_\ell\|^2_{L^2({\mathbb R})}+\|\nabla f^k_\ell\|^2_{L^2({\mathbb R})} +   \left\| \tau_+^{s/2} \tau_-^{s/2}F^k_\ell\right\|^2_{L^2((0,T)\times \mathbb{R})}
\end{aligned}\right.
\end{equation}
This estimate and the lower bounds \eqref{eq.esc4} imply  \eqref{CEs1} and complete the proof.

In the case $-1  < \beta < 1$ we have pentagon $\mathrm{pen}(\alpha,\beta,T)=ABCDE$  on Figure \ref{fig:f2} and we follow the same argument and deduce
\begin{equation}\label{eq.esc2m}
  \left\{  \begin{aligned}
    E_-(\beta,T) + E_+(\alpha,T) + E_0(\alpha,\beta,T) + \iint_{\mathrm{pen}(\alpha,\beta,T)}\mathfrak{R}|u^k_\ell|^2 \lesssim \\ \lesssim  \|g^k_\ell\|^2_{L^2({\mathbb R})}+\|\nabla f^k_\ell\|^2_{L^2({\mathbb R})} +   \left\| \tau_+^{s/2} \tau_-^{s/2}F^k_\ell\right\|^2_{L^2((0,T)\times \mathbb{R})}
\end{aligned}\right.
\end{equation}

\begin{figure}[htb!]
\centering%
\includegraphics[scale=0.4]{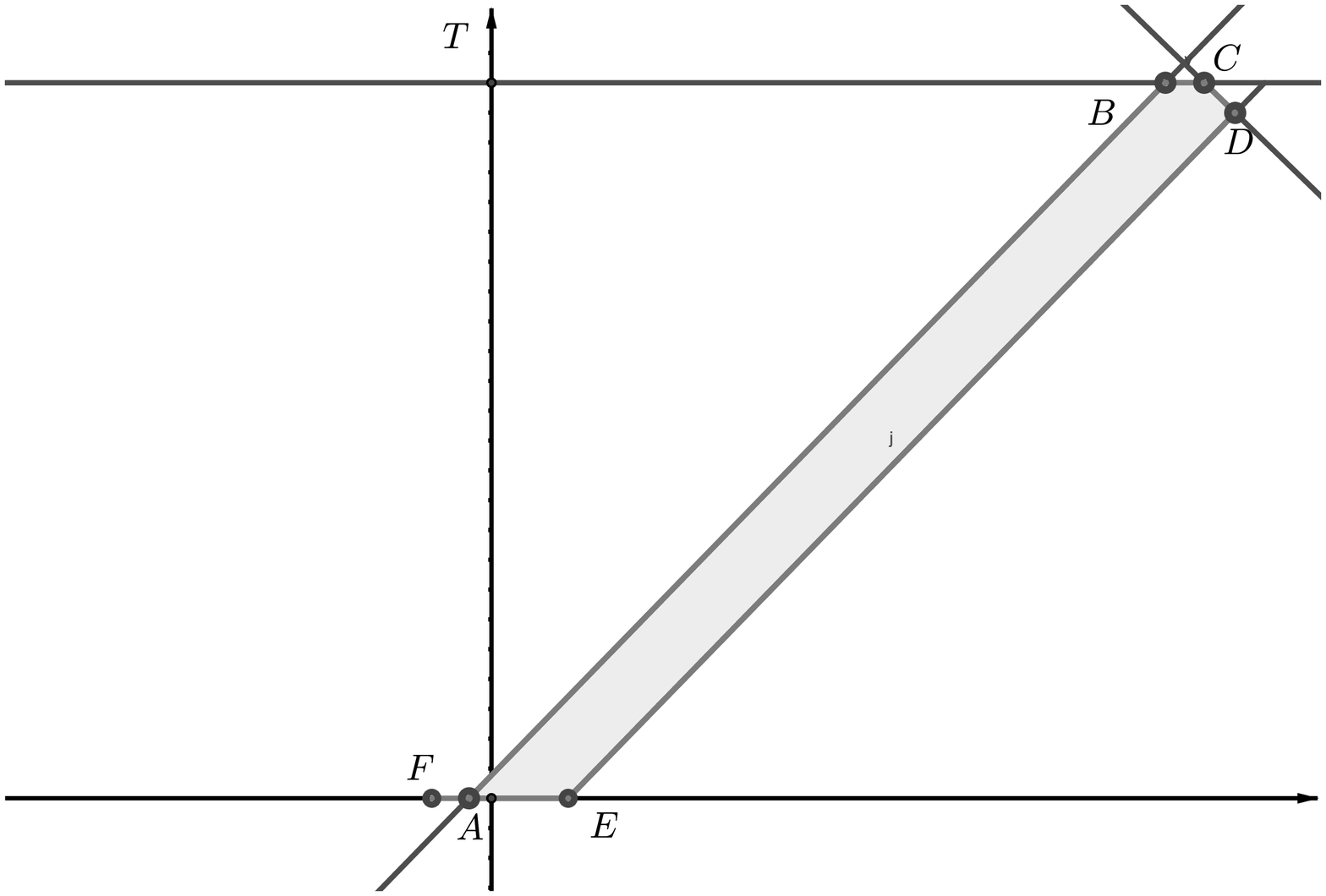}
\caption{Domain of integration for $\beta \in (-1,1)$}
\label{fig:f2}
\end{figure}

\end{proof}

\begin{coro}
Let $s\in (1,2)$, $\theta\in [0,1]$ and $q\in [2,\infty]$.
Define $\sigma\in [2, \infty]$, $\kappa\in [2,q]$ by
$$
\frac{1}{\sigma}=\frac{\theta}{2}+\frac{1-\theta}{\infty}, \ \
\frac{1}{\kappa}
=\frac{\theta}{q}+\frac{1-\theta}{2}.
$$
Then, under the assumptions of Proposition 3.2,
we have
\begin{Eq} \label{BE}
& \|r^{(1-3\theta)/2} (2+t+r)^{(1-\theta)/2}(2+t-r)^{(s-1+\theta)/2}
u(t,r\omega)\|_{L^\sigma_r L^\kappa(S^{2})}
\\
\lesssim \, &
\|g\|_{L^2({\mathbb R}^3)}+\|\nabla f\|_{L^2({\mathbb R}^3)}
\\ &
+\|(2+t+r)^{s/2}(2+t-r)^{(1+\delta)/2} F\|_{L^2([0,T] \times {\mathbb R}^3)}
\end{Eq}
for $t \in [0,T)$. Here and in what follows,
for $\sigma \in [1,\infty)$, $q\in [1,\infty]$, we denote
$$
\| f\|_{L^\sigma_r L^q(S^{2})}
:=\left(\int_0^\infty \|f(r\omega)\|_{L^q(S^2)}^\sigma r^2 dr \right)^{1/\sigma}.
$$
\end{coro}

\begin{proof}
For $q\in [2, \infty)$ and $\phi\in H^1(S^2)$, we have
$$
\|\phi\|_{L^q(S^2)}\lesssim \|\phi\|_{H^1(S^2)}
\sim \sum_{|\alpha| \le 1} \| R^\alpha \phi\|_{L^2(S^2)}
$$
by the Sobolev embedding theorem. Therefore, \eqref{CE} implies
\begin{Eq} \label{CE1}
& \|r^{-1} (2+t-r)^{s/2}
u(t,r\omega)\|_{L^2_r L^q(S^{2})}
\\
\lesssim \, &
\|g\|_{L^2({\mathbb R}^3)}+\|\nabla f\|_{L^2({\mathbb R}^3)}
\\ &
+\|(2+t+r)^{s/2}(2+t-r)^{(1+\delta)/2} F\|_{L^2([0,T] \times {\mathbb R}^3)}
\end{Eq}
for $t \in [0,T)$.
Combining \eqref{CE} and Proposition \ref{Lai} with $n=3$, we get
\begin{Eq} \label{CE2}
& \|r^{1/2} (2+t+r)^{1/2}(2+t-r)^{(s-1)/2}
u(t,r\omega)\|_{L^\infty_r L^2(S^{2})}
\\
\lesssim \, &
\|g\|_{L^2({\mathbb R}^3)}+\|\nabla f\|_{L^2({\mathbb R}^3)}
\\ &
+\|(2+t+r)^{s/2}(2+t-r)^{(1+\delta)/2} F\|_{L^2([0,T] \times {\mathbb R}^3)}
\end{Eq}
for $t \in [0,T)$.
The interpolation between \eqref{CE1} and \eqref{CE2}
yields the desired estimate.
\end{proof}

\section{Proof of Theorem \ref{main}}

We introduce the following function space $X_T$ in which
we shall look for the solution to \eqref{pn}:
\begin{Eq*}
X_T:=\{ & v \in\bigcap_{j=0}^1 C^j([0,T]:{\mathcal D}(A^{\frac{2-j}2})) |\,
\|v\|_{X_T}<\infty,\,
\supp v(t,\cdot) \subset B(1+t) \},
\end{Eq*}
where we set
\begin{Eq}\label{norm1}
\|v\|_{X_T}:=\sup_{t \in [0,T]} \{  \sum_{j=0}^2
\|A^{\frac{2-j}2} v(t)\|_{L^2({\mathbb R}^3)}
+\sum_{j=1}^2
\|A^{\frac{2-j}2} \partial_t v(t)\|_{L^2({\mathbb R}^3)}
\}.
\end{Eq}
In order to handle the nonlinear term $F(u(t))=b|u(t)|^{p-1}u(t)$, we prepare the following estimates.

\begin{lemma}
Let $1<p<3$.
For $u$, $v\in X_T$ and $t \in [0,T]$, we have
\begin{Eq} 
 \|F(u(t))-F(v(t))\|_{L^2({\mathbb R}^3)}
 \label{diff1}
 \lesssim   (\|u\|_{X_T}^{p-1}
+\|v\|_{X_T}^{p-1})
\|u-v\|_{X_T},
\end{Eq}
\begin{Eq}
 \label{diff3}
\|
A^{\frac12} (F(u(t))-F(v(t)))\|_{L^2({\mathbb R}^3)}
 \lesssim
 (\|u\|_{X_T}^{p-1}+\|v\|_{X_T}^{p-1}) \|u-v\|_{X_T}.
\end{Eq} 
\end{lemma}

\begin{proof}
It follows that
\begin{Eq*}
& \| F(u(t))-F(v(t))\|_{L^2({\mathbb R}^3)}
\\
\lesssim & (\|u(t)\|_{L^{2p} ({\mathbb R}^3) }^{p-1}
+\|v(t)\|_{L^{2p}({\mathbb R}^3)}^{p-1} )
\|(u-v)(t)\|_ {L^{2p}({\mathbb R}^3)}.
\end{Eq*}
By the Gagliardo-Nirenberg inequality, we have
\begin{Eq*}
\|\varphi\|_{L^{2p}({\mathbb R}^3)}
\lesssim
\|\varphi\|_{L^{2}({\mathbb R}^3)}^{1-\gamma}
\|\nabla \varphi\|_{L^{2}({\mathbb R}^3)}^\gamma
\lesssim
\|\varphi\|_{L^{2}({\mathbb R}^3)}+
\|\nabla \varphi\|_{L^{2}({\mathbb R}^3)},
\end{Eq*}
where we put $\gamma=3(p-1)/(2p)\in (0,1)$
for $1<p< 3$.  
Therefore, we get \eqref{diff1}
by \eqref{equivm1}.

Noting that
\begin{Eq*}
& \|\nabla (F(u(t))-F(v(t)))\|_{L^2({\mathbb R}^3)}
\\
\lesssim & \|u(t)\|_{L^{2p} ({\mathbb R}^3) }^{p-1}
\|\nabla (u-v)(t)\|_ {L^{2p}({\mathbb R}^3)}
\\
 &
+\|\nabla v(t)\|_ {L^{2p}({\mathbb R}^3)}
(\|u(t)\|_{L^{2p}({\mathbb R}^3)}^{p-2}
+\|v(t)\|_{L^{2p}({\mathbb R}^3)}^{p-2})
\|(u-v)(t)\|_ {L^{2p}({\mathbb R}^3)}
\end{Eq*}
and
\begin{Eq*}
\|\nabla \varphi\|_{L^{2p}({\mathbb R}^3)}
\lesssim & \sum_{j=1}^3
\|\partial_j \varphi\|_{L^{2p}({\mathbb R}^3)}
\lesssim
 \sum_{j=1}^3 (
\|\partial_j \varphi\|_{L^{2}({\mathbb R}^3)}+
\|\nabla \partial_j \varphi\|_{L^{2}({\mathbb R}^3)})
\\
\lesssim
&
\|\nabla \varphi\|_{L^{2}({\mathbb R}^3)}+
\| \Delta \varphi\|_{L^{2}({\mathbb R}^3)}
\lesssim
\|A^\frac12 \varphi\|_{L^{2}({\mathbb R}^3)}+
\| A \varphi\|_{L^{2}({\mathbb R}^3)},
\end{Eq*}
we obtain \eqref{diff3}
by \eqref{equivm1} and \eqref{upper}.
This completes the proof.
\end{proof}

For a given  $(f,g)\in {\mathcal D}(A)\times {\mathcal D}(A^{\frac12})$ satisfying
$\supp f$, $\supp g\subset B(1)$, we define a sequence 
by
\begin{Eq}\label{pni}
u_{m+1}(t)=&u_0(t)
+\int_0^t \frac{\sin (t-\tau) \sqrt{A}}{\sqrt{A}} F(u_m(\tau)) d\tau
\   (m=0,1,\dots),
\\
u_0= & (\cos t \sqrt{A} ) f+\frac{\sin t\sqrt{A}}{\sqrt{A}} g.
\end{Eq}

\begin{thm} \label{local}
Let $1<p<3$.
Assume $(f,g)\in {\mathcal D}(A)\times {\mathcal D}(A^{\frac12})$ satisfies
$\supp f$, $\supp g\subset B(1)$.
Set
\begin{Eq}
\eta:=
\sum_{j=0}^2
\|A^{\frac{2-j}2} f\|_{L^2({\mathbb R}^3)}
+\sum_{j=1}^2
\|A^{\frac{2-j}2} g\|_{L^2({\mathbb R}^3)}.
\end{Eq}
Then there exists $T=T(\eta)>0$
such that \eqref{pn} admits a unique strong solution $u(t)$ in $X_{T}$.

Moreover, if $T_*$ is the supremum over all such $T$, then $T_*=\infty$ or
$\|u(t)\|_{L^{2p}({\mathbb R}^3)} \not\in L^\infty([0,T_*))$.
\end{thm}

\begin{proof}
Due to \eqref{diff1} and \eqref{diff3}, the existence and uniqueness
follow from the standard argument.

In order to prove the blow-up criterion,
by assuming $T_*<\infty$
and there exists a constant $M$ such that
$\sup_{t \in [0,T_*)} \|u(t)\|_{L^{2p}({\mathbb R}^3)} \le M$,
we shall derive a contradiction.
In the following, we always assume $t \in [0,T_*)$.
In view of the proof of \eqref{diff1} and \eqref{diff3} as $v=0$, we get
\begin{Eq*} 
 \|F(u(t))\|_{L^2({\mathbb R}^3)}
 +\|A^{\frac12} F(u(t))\|_{L^2({\mathbb R}^3)}
 \lesssim  \|u(t)\|_{L^{2p} ({\mathbb R}^3)}^{p-1}  E(t),
\end{Eq*}
where we set
$$
E(t):= \sum_{j=0}^2
\|A^{\frac{2-j}2} u(t)\|_{L^2({\mathbb R}^3)}
+\sum_{j=1}^2
\|A^{\frac{2-j}2} \partial_t u(t)\|_{L^2({\mathbb R}^3)}.
$$
Thus, using the integral equation for the solution:
\begin{Eq*}
u(t)=u_0(t)
+\int_0^t \frac{\sin (t-\tau) \sqrt{A}}{\sqrt{A}} F(u(\tau)) d\tau
\end{Eq*}
we obatin
$$
E(t) \le C_0 (1+T_*) \eta+C_1 (1+T_*)  \int_0^t  M^{p-1} E(\tau) d\tau
$$
with some universal constants $C_0$, $C_1$.
By the Gronwall inequality, we get
\begin{Eq*}
E(t) \le  C_0(1+T_*)\eta \exp(C_1 M^{p-1}  (1+T_*)  T_*),
\end{Eq*}
which implies $\sup_{t \in [0,T_*)} E(t) <\infty$.
Therefore, we can extend $u(t)$ at $t=T_*$ so that
$(u(T_*), (\partial_t u)(T_*) )\in {\mathcal D}(A) \times
{\mathcal D}(A^{\frac{1}2})$.
But this means that the local solution $u(t)$ can be extended as a solution
of \eqref{pn} beyond $T_*$.
This contradcts the definition of $T_*$ and finishes the proof.
\end{proof}

In order to prove the global existence, we assume $p>1+\sqrt{2}$,
Then, since  $p^2-2p-1>0$ for $p>1+\sqrt{2}$, we can choose
$\delta>0$ 
so small that $p^2-2p-1>p\delta/2$.
For such $\delta$, we shall take $\theta\in(0,1)$ so that
$p^2\theta <(p+1)(p-2)$ and
$2p\theta<\delta$.
We define $\sigma=2/\theta$ and take $\kappa\in (2,\infty)$.
Now, we introduce the following quanity:
\begin{Eq*} 
\|\hspace{-0.5mm} | v \|\hspace{-0.5mm} |_{X_T}=
\sup_{t \in [0,T]} \{
\|w(t)  v(t)\|_{L^\sigma_r L^{\kappa}(S^2)}
+\sum_{j=1}^3 \|w(t) R_j v(t)\|_{L^\sigma_r L^\kappa(S^2)} \}.
\end{Eq*}
where the weight function $w=w(t,r)$ is defined by
\begin{Eq*}
w(t,r):=
r^{(1-3\theta)/2} \tau_+^{(1-\theta)/2} \tau_-^{(s-1+\theta)/2},
\quad  s=1+(2/p).
\end{Eq*}
In order to make use of the estimate of conformal type,
we prepare the following lemma.

\begin{lemma}
Assume $1+\sqrt{2}<p<3$.
Let $\delta$, $\theta$, $\sigma$, $\kappa$ and $s$ be as in the above.
Then, for $u\in X_T$ satisfying $\|\hspace{-0.5mm} | v \|\hspace{-0.5mm} |_{X_T}<\infty$,
we have
\begin{Eq}
 \label{diff2}
& \|\tau_+^{s/2} \tau_-^{(1+\delta)/2}F(u)\|_{L^2_T L^2({\mathbb R}^3)}
+ \sum_{j=1}^3
\|\tau_+^{s/2} \tau_-^{(1+\delta)/2}
R_j F(u)\|_{L^2_T L^2({\mathbb R}^3)}
\\ 
& \lesssim  \|\hspace{-0.5mm} | v \|\hspace{-0.5mm} |_{X_T}^p.
\end{Eq} 
\end{lemma}

\begin{proof}
Since $\kappa>2$, we have  
\begin{Eq*}
\|\phi\|_{L^\infty(S^2)} \lesssim
\|\phi\|_{W^{1,\kappa}(S^2)}
\sim \| \phi\|_{L^\kappa(S^2)}
+ \sum_{j=1}^3
\|R_j  \phi\|_{L^\kappa(S^2)},
\end{Eq*}
by the Sobolev inequality.
Therefore, we can deduce
\begin{Eq*}
\| F(u(t,r))\|_{L^2(S^2)}
+ \sum_{j=1}^3
\|R_j F(u(t,r)) \|_{L^2(S^2)} \lesssim
\|u(t,r)\|_{W^{1,\kappa}(S^2)}^{p}.
\end{Eq*}
Thus, noting that $(\sigma/2p)^{-1}+(1/(1-p\theta))^{-1}=1$, we get
\begin{Eq} \label{est1}
& \| \tau_+^{s/2} \tau_-^{(1+\delta)/2} F(u(t))\|_{L^2({\mathbb R}^3)}
+ \sum_{j=1}^3 \|\tau_+^{s/2} \tau_-^{(1+\delta)/2} R_j F(u(t))\|_{L^2({\mathbb R}^3)}
\\
&  \lesssim
\|w(t) u(t)\|_{L_r^\sigma W^{1,\kappa}(S^2) }^{p}
  (J(t))^{(1-p\theta)/2},
\end{Eq}
where we set
\begin{Eq*}
  J(t)
:= & \int_0^{t+1}
( \tau_+^{s} \tau_-^{1+\delta} (w(t,r))^{-2p})^{\frac{1}{1-p\theta}} \, r^2dr
 \\
 = & \int_0^{t+1}
    r^{2-\frac{p(1-3\theta)}{1-p\theta} }
    \tau_+^{\frac{s-p(1-\theta)}{1-p\theta}}
  \tau_-^{\frac{1+\delta-p(s-1+\theta)}{1-p\theta}}
   dr.
\end{Eq*}
Since
\begin{Eq*}
& s-p(1-\theta)=1+(2/p)-p+p\theta
=(-(p+1)(p-2)+p^2\theta)/p<0,
\\
& 1+\delta-p(s-1+\theta)=-1+\delta-p\theta,
\end{Eq*}
we have
\begin{Eq*}
 J(t)
\lesssim    (2+t)^{ \frac{1+(2/p)-p+p\theta}{1-p\theta}} \int_0^{t+1}
  r^{2-\frac{p(1-3\theta)}{1-p\theta} }
 (2+t-r)^{-\frac{1-\delta+p\theta}{1-p\theta}} dr.
\end{Eq*}
Noting that
\begin{Eq*}
3-\frac{p(1-3\theta)}{1-p\theta}>0 \ \  \mbox{if}\  p<3
\ \ \mbox{and}\ \
1-\frac{1-\delta+p\theta}{1-p\theta}>0 \ \  \mbox{if}\ 2p\theta<\delta,
\end{Eq*}
we see that the above $r$-integral is bounded by
\begin{Eq*}
 &  (2+t)^{-\frac{1-\delta+p\theta}{1-p\theta}}
\int_0^{(t/2)+1} r^{2-\frac{p(1-3\theta)}{1-p\theta} }dr
\\
& \quad
+(2+t)^{2-\frac{p(1-3\theta)}{1-p\theta} }
\int_{(t/2)+1}^{t+1}
 (2+t-r)^{-\frac{1-\delta+p\theta}{1-p\theta}} dr
\lesssim
(2+t)^{\frac{2+\delta-p\theta-p}{1-p\theta}}.
\end{Eq*}
Thus we get $
  (J(t))^{1-p\theta}  \lesssim (2+t)^{3+(2/p)-2p+\delta}.
$
Now, it follows from \eqref{est1} that
\begin{Eq*}
& \|\tau_+^{s/2} \tau_-^{(1+\delta)/2}F(u)\|_{L^2_T L^2({\mathbb R}^3)}
+ \sum_{j=1}^3
\|\tau_+^{s/2} \tau_-^{(1+\delta)/2}
R_j F(u)\|_{L^2_T L^2({\mathbb R}^3)}
\\ 
& \lesssim  \|\hspace{-0.5mm} | v \|\hspace{-0.5mm} |_{X_T}^p
\left( \int_0^T  (2+t)^{3+(2/p)-2p+\delta} dt \right)^{1/2}
 \lesssim  \|\hspace{-0.5mm} | v \|\hspace{-0.5mm} |_{X_T}^p,
\end{Eq*}
because $4+(2/p)-2p+\delta<0$ by the choice of $\delta$ under the assumption
$p > 1+\sqrt{2}$.
Hence \eqref{diff2} holds.
This completes the proof.
\end{proof}

\noindent
{\bf End of the proof of Theorem \ref{main}}:\
By the blow-up criterion given in Theorem \ref{local},
we only need to derive a uniform bound of $\|u(t)\|_{L^{2p}({\mathbb R}^3)}$
with respect to ${t \in [0,T_*)}$.

Let $0<T<T_*$.
Seeing the proof of \eqref{diff2}, we find
\begin{Eq*}
& \|u(t)\|_{L^{2p}({\mathbb R}^3)}
\le \| \tau_+^{s/2} \tau_-^{(1+\delta)/2} |u(t)|^p \|_{L^2({\mathbb R}^3)}^{1/p}
\\
&  \lesssim
\|w(t) u(t)\|_{L_r^\sigma W^{1,\kappa}(S^2) }
  (J(t))^{(1-p\theta)/(2p)} \lesssim
  \|\hspace{-0.5mm} | u \|\hspace{-0.5mm} |_{X_T}^p,\  {t \in [0,T]}.
\end{Eq*}
Therefore, it suffices to get a uniform bound of
$ \|\hspace{-0.5mm} | u \|\hspace{-0.5mm} |_{X_T}$
with respect to ${T \in [0,T_*)}$.
Since $u$ satisfies \eqref{linear}
with $F=F(u)$, applying the rotational vector fields $R_j$ to
the equation, and using
\eqref{BE} and \eqref{diff2}, we obtain
\begin{Eq*}
\|w(t) u(t)\|_{L^\sigma_rL^\kappa(S^2)}
+\sum_{j=1}^3
\|w(t) R_j u(t)\|_{L^\sigma_rL^\kappa(S^2)}
\le C_0\eta+C_1 \|\hspace{-0.5mm} | u \|\hspace{-0.5mm} |_{X_T}^p
\end{Eq*}
for $t \in [0,T)$, because
$\|R_j g\|_{L^2({\mathbb R}^3)}
+\|\nabla R_j f\|_{L^2({\mathbb R}^3)} \lesssim \eta$
by the assumption that the supports of $f$ and $g$ are compactly supported.
Thus we get
\begin{Eq*}
\|\hspace{-0.5mm} | v \|\hspace{-0.5mm} |_{X_T}
\le C_0\eta +C_1
\|\hspace{-0.5mm} | u \|\hspace{-0.5mm} |_{X_T}^p
\end{Eq*}
for any ${T \in [0,T_*)}$.
This means that there exists $\eta_0>0$ such that
if $0<\eta<\eta_0$, then
$\|\hspace{-0.5mm} | u \|\hspace{-0.5mm} |_{X_T}\le 2C_0\eta$
holds for ${T \in [0,T_*)}$, because $p>1$.
Hence the proof of Theorem \ref{main} is completed.
\hspace{25mm} $\Box$


\bibliographystyle{plain}
\bibliography{preprint_STR_pot_GK_17_01_22}
\end{document}